\documentclass[11pt,a4paper]{amsart}
\usepackage{mathrsfs,amsthm,amssymb,amsmath,url}

\theoremstyle{plain}
    \newtheorem{thm}{Theorem}

    \newtheorem{lem}[thm]{Lemma}
    
    \newtheorem{cor}[thm]{Corollary}
    \newtheorem{fact}[thm]{Fact}
    
    \newtheorem*{prob}{Problem}
    
   \newtheorem*{ans}{Answer}

\theoremstyle{definition}
    \newtheorem{defn}[thm]{Definition}
    
\theoremstyle{remark}

\newcommand{\nin}{\notin}

\newcommand{\on}{{\upharpoonright}}

\renewcommand{\ss}{\subseteq}
\newcommand{\inv}{^{-1}}
\newcommand{\To}{\rightarrow}
\newcommand{\mult}{\times}

\DeclareMathOperator{\Cl}{Cl}
\DeclareMathOperator{\Pol}{Pol}

\DeclareMathOperator{\ran}{ran}
\DeclareMathOperator{\dom}{dom}

\newcommand{\C}{{\mathscr C}}
\newcommand{\M}{{\mathscr M}}
\newcommand{\N}{{\mathbb N}}

\renewcommand{\O}{{\mathscr O}}
\newcommand{\On}{{\mathscr O}^{(n)}}
\newcommand{\Oo}{{\mathscr O}^{(1)}}
\newcommand{\un}{^{(n)}}

\newcommand{\linen}{\omega\mult\{n\}}
\newcommand{\cc}{_{\cup c}}
\newcommand{\cco}{_{\cup c^1}}
\newcommand{\ccup}{\bigcup_{c\in X^S}}

\author
{Martin Goldstern}
\address{Algebra\\TU Wien\\Wiedner Hauptstrasse 8-10/104\\A-1040 Wien, Austria}
\email{goldstern@tuwien.ac.at}\urladdr{http://www.tuwien.ac.at/goldstern/}
\author
{Michael Pinsker}
\address{Laboratoire de Math\'{e}matiques Nicolas Oresme\\
CNRS UMR 6139\\ Universit\'{e} de Caen\\14032 Caen Cedex\\
France}
\email{marula@gmx.at}\urladdr{http://dmg.tuwien.ac.at/pinsker/}
\thanks{The second author is grateful for support through
grant P17812 as well as through the Erwin Schr\"{o}dinger Fellowship
J2742-N18 of the Austrian Science Foundation (FWF)}

\title{Ideal clones: Solution to a problem of Cz\'{e}dli and Heindorf}



\begin{document}

\begin{abstract}
    Given an infinite set $X$ and an ideal $I$ of subsets of $X$, the set of
    all finitary operations on $X$ which map all (powers of) $I$-small
    sets to $I$-small sets is a clone. In \cite{CH01}, G.~Cz\'{e}dli
    and L.~Heindorf asked whether or not for two particular ideals
    $I$ and $J$ on a countably infinite set $X$, the corresponding
    ideal clones were a covering in the lattice of clones. We give an
    affirmative answer to this question.
\end{abstract}

\maketitle

\section{Clones and ideals}

Let $X$ be an infinite set and denote the set of all $n$-ary
operations on $X$ by $\On$. Then $\O:=\bigcup_{n\geq 1}\O\un$ is the
set of all finitary operations on $X$. A subset $\C$ of $\O$ is
called a \emph{clone} iff it contains all projections, i.e. for all
$1\leq k\leq n$ the function $\pi^n_k\in\On$ satisfying
$\pi^n_k(x_1,\ldots,x_n)=x_k$, and is closed under composition. The
set of all clones on $X$, ordered by set-theoretical inclusion,
forms a complete algebraic lattice $\Cl(X)$. The structure of this
lattice has been subject to much investigation, many results of
which are summarized in the recent survey \cite{GP06survey}.

One such result, from \cite{Ros76}, states that there exist as many
dual atoms (``\emph{precomplete clones}'') in $\Cl(X)$ as there are
clones (that is, $2^{2^{|X|}}$), suggesting that it is impossible to
describe all of them (as opposed to the clone lattice on finite $X$,
where the dual atoms are finite in number and explicitly known
\cite{Ros70}). Much more recently, a new and short proof of this
fact was given in \cite{737}. It was observed that given an ideal
$I$ of subsets of $X$, one can associate with it a clone $\C_I$
consisting of those operations $f\in\O$ which satisfy $f[A^n]\in I$
for all $A\in I$. The authors then showed that prime ideals
correspond to precomplete clones, and that moreover the clones
induced by distinct prime ideals differ, implying that there exist
as many precomplete clones as prime ideals on $X$; the latter are
known to amount to $2^{2^{|X|}}$.

The study of clones that arise in this way from ideals was pursued
in \cite{CH01}, for countably infinite $X$. The authors concentrated
on the question of which ideals induce precomplete clones, and
obtained a criterion for precompleteness. In the same paper, three
open problems were posed, and we provide the solution to their
second problem in this article.

We mention that in the article \cite{BGHP} which is still in
preparation, a theory of clones which arise from ideals is being
developed. In particular, that paper contains the solution to the
first problem from \cite{CH01}, which asked whether every ideal
clone could be extended to a precomplete ideal clone. Its third
question, which asks whether every ideal clone is covered by another
ideal clone, is still unsolved.

\section{The question and its answer}

The problem from \cite{CH01} we are going to solve concerns two
particular ideals $I, J$ on $X=\omega\mult\omega$: We call sets of
the form $\linen$, where $n\in\omega$, \emph{lines}, and sets of the
form $\{n\}\times\omega$ \emph{rows}. $J$ is the ideal of those
subsets of $X$ whose intersection with every line is finite. The
\emph{width} of a subset $Y$ of $X$ is defined by $\sup\{|Y\cap
(\linen)|:n\in\omega\}$. $I$ is the ideal of those sets which have
finite width. Clearly, $I\subseteq J$.

It is easy to see that distinct proper ideals containing all finite
subsets of $X$ yield distinct clones, so $\C_I\neq \C_J$. In
general, the mapping $K\mapsto \C_K$ is not monotone; in this case,
however, we have $\C_I\subsetneq \C_J$. This was shown in
\cite{CH01} in order to provide the first example of an ideal
(namely, $I$) which does not induce a precomplete clone. (It also
follows easily from the more recent and general results in
\cite{BGHP}, since $J$ is the \emph{regularization} of $I$.)

The second problem posed in \cite{CH01} was

\begin{prob}
    Is $\C_J$ a cover of $\C_I$ in $\Cl(\omega\mult\omega)$? That
    is, is the interval $(\C_I,\C_J)$ of $\Cl(\omega\mult\omega)$ empty?
\end{prob}

We will now prove that the answer is

\begin{ans}
    Yes.
\end{ans}

We remark that our original motivation for working on this problem
was the fact that a negative answer would have yielded a negative
answer to the more general third problem from \cite{CH01}, asking
whether every ideal clone is covered by an ideal clone. This is
because one can show, and quite easily so with the methods from
\cite{BGHP}, that if $\C_I$ has a cover which is an ideal clone,
then this cover must be equal to $\C_J$.

\section{The proof}

In order to prove that $\C_J$ covers $\C_I$, we must show that if we
are given any $g\in\C_J$ and any $f\in \C_J\setminus\C_I$, then $g$
can be written as a term over the set $\{f\}\cup \C_I$. We divide
our proof into two parts: In the first part, we show that the
$m$-ary function $g$ can be decomposed as $g=g'(h_1,\ldots,h_m)$,
where the $h_i$ are $m$-ary operations in $\C_I$, and the $m$-ary
operation $g'\in\C_J$ is what we will call \emph{hereditarily
thrifty}; this will be achieved in
Corollary~\ref{cor.decomposition}. In the second part, we show that
every hereditarily thrifty operation in $\C_J$, so in particular
$g'$, can be written as a term over $\{f\}\cup \C_I$ (Lemma
\ref{lem.hereditarilyThriftyAreGenerated}).

Before we go into Part 1, we fix some notation and give the
definition of a hereditarily thrifty function.

From now on, we write $X=\omega\times\omega$. For $d\in X$, we write
$d=(d^x|d^y)$ and refer to $d^x$ as the \emph{$x$-coordinate} and to
$d^y$ as the \emph{$y$-coordinate} of $d$. We set $M=\{1,\ldots,m\}$
for all $m\geq 1$. An $m$-ary function on $X$ maps the $M$-tuples of
$X^M$ into $X$.

For notational reasons we will often consider partial functions from
$X^M$ into a set $Y$, that is, functions into $Y$ whose domain is a
subset of $X^M$. We refer by $\dom(p)$ to the domain and by
$\ran(p)$ to the range of a partial function $p$.

\begin{defn}
    Let $\C$ be a clone. For a partial function $p$ we write
    $p\in \C$ iff there is a total function $p'\in \C$ extending $p$.

    If $p:X^M\to X^M$ is a partial function, then we write $p\in \C$ iff
    each of the $k$ component functions $\pi^m_k\circ p$ of $p$ is in $\C$.
\end{defn}

\begin{defn}
    For any partial function $p:X^M\to X$, let
    $\bar p$ be defined
        as the extension of $p$ to $X^M$ obtained by setting $\bar p(x) =(0|0)$ for all $x\in X^M\setminus\dom(p)$.
\end{defn}

The following lemma justifies the use of partial functions in our
proof, since we can always extend them to total functions respecting
$I$.

\begin{lem}\label{lem.extension}
    For all partial functions $p: X^M\To X$ we have $p\in \C_I$ iff $\bar p\in \C_I$.
\end{lem}
\begin{proof}
    Obvious.
\end{proof}

\begin{defn}
    For every $m\geq 1$ and every $k\in\omega$, we define a subset
    $B^M_k\subseteq X^M$ to consist of all tuples of $X^M$ which have at
    least one component whose $y$-coordinate is less than $k$.
    Formally, if for all $i\in M$ and all $M$-tuples $u\in X^M$ we write $u_i$ for the $i$-th component of $u$,
    $$B^M_k:=\{u\in X^M:\exists i\in M\, ((u_i)^y<k)\}.$$
\end{defn}

\begin{defn}
    We call a subset of $X^M$ \emph{bounded} iff it is contained in
    $B^M_k$ for some $k\in\omega$. A partial function  $p:X^M\to Y$ is \emph{thrifty} iff
    $p\inv[d]$ is bounded for all $d\in Y$.
\end{defn}

Observe that when we talk about subsets of $X$ having \emph{finite
width}, we want to restrict the possible \emph{x-coordinates} for
each \emph{y} to a small set (of \emph{size} $<k$); for a
1-dimensional \emph{bounded} subset of $X$, on the other hand, we
restrict the possible \emph{y-coordinates} for each \emph{x} to a
small set (with \emph{maximum} $<k$).

 We will now give the definition of a hereditarily thrifty function.
Loosely speaking, a function is hereditarily thrifty iff it is
thrifty and whenever we fix some of its arguments to fixed values in
$X$, then the function of smaller arity obtained this way is thrifty
as well. The formal definition is as follows:

\begin{defn}
Let $M$ be a finite index set, and let $S\subseteq M$. Set $T:=
M\setminus S$. An $S$-tuple is a total function from $S$ to $X$. The
$M$-tuples are exactly the unions of $S$-tuples with $T$-tuples.
(Since $X^S \times X^T$ is naturally isomorphic with $X^M$ through
the map $(c,z)\mapsto c\cup z$, we may occasionally identify $X^S
\times X^T$ with $X^M$, or the tuple $c\cup z$ with the ordered pair
$(c,z)$.)

For any partial function $p:X^M \to Y$,  we write $p\cc$ for the
    function $p\cc:X^{T}\to Y$ defined by $p\cc(z) = p(z\cup c)$.
\end{defn}

\begin{defn}
    Let $p: X^M\To Y$ be a partial function. We
    call $p$ \emph{hereditarily thrifty} iff $p\cc$ is thrifty for
    all $S\subseteq M$ and all tuples $c\in X^S$.
\end{defn}

\bigskip \subsection*{Part 1:  Decomposing $\C_J$--functions into hereditarily thrifty
functions and $\C_I$--functions}

In this part, we show that given any function $g: X^M\To Y$, we can
find a function $h: X^M\To X^M$, $h\in\C_I$, and a hereditarily
thrifty $g'\subseteq g$ such that $g=g'\circ h$ (Corollary
\ref{cor.decomposition}). Clearly, if $Y=X$ and $g\in \C_J$, then
also $g'\in\C_J$; thus, we can decompose the $\C_J$--function $g$
into the composition of a hereditarily thrifty $\C_J$--function $g'$
with a $\C_I$--function $h$.

We first state some trivial facts about the composition of
functions. In order to avoid confusion with the components of
elements of $X$, we use the symbol $\mathbb{N}$ (rather than
$\omega$) to denote the index set of countable families of
functions.

\begin{fact}\label{fact.union}
    Assume that $f_n$, $n\in \N$, are functions with pairwise disjoint
    domains, and  similarly $g_n$, $n\in \N$. Then:
    \begin{itemize}
        \item  $\bigcup_n f_n $  and $\bigcup_n g_n $  are
        functions.
        \item  The functions $f_n\circ g_n$ (for $n\in \N$) have pairwise
        disjoint domains,  and hence their union is a function.
        \item
        $\bigcup_n \bigg(f_n\circ g_n\bigg)\subseteq \bigg(\bigcup_n f_n
        \bigg) \circ \bigg(\bigcup_n g_n\bigg)$.
    \end{itemize}
\end{fact}

\begin{fact} \label{fact.sub}
    If $g \subseteq g' \circ h'$, then
    there is a function $h\subseteq h'$ such that $g=g'\circ h$.
\end{fact}

\begin{defn}
    We call a partial function $p:X^M\to Y$ \emph{wasteful} iff for
    every $y\in \ran(p)$ the set $p\inv[y]$ is unbounded.
\end{defn}

\begin{lem}\label{lem.splitting}
    For every partial function $p:X^M\to Y$ we can find disjoint sets
    $W_p, T_p \subseteq \dom(p)$  such that $\dom(p) = W_p\cup T_p$, and
    $p\on W_p$ is wasteful and $p\on T_p$ is thrifty.
\end{lem}
\begin{proof}
    Easy.
\end{proof}

\begin{lem}\label{lem.finiteUnions}
    Let $p_1$, $p_2$ be partial functions on $X^M$ with disjoint
    domains. Then $p_1\cup p_2\in \C_I$ iff $p_1\in \C_I$ and $p_2\in
    \C_I$. Moreover,  $p_1\cup p_2$ is thrifty iff both $p_1$ and $p_2$
    are thrifty.
\end{lem}
\begin{proof}
    This is clear since $I$ is closed under finite
    unions.
\end{proof}

Recall that we defined the width of a subset $A\subseteq X$ as
$\sup\{|A\cap (\linen)|:n\in\omega\}$. We extend this definition to
subsets of $X^M$.

\begin{defn}
    The \emph{width} of a set $A\subseteq X^M$ is the maximum of the
    widths of its projections onto the components of $M$.
\end{defn}

\begin{lem}\label{countable}
    Let $g_n:X^M\to Y$, $n\in \mathbb{N},$ be wasteful partial functions. Then we can find
    \begin{itemize}
        \item a set $A \subseteq X^M$ of width 1
        \item a family of injective functions $g_n'\subseteq g_n$, with $A$ being
        the disjoint union of the sets $\dom(g'_n)$
        \item   partial functions $h_n:X^M\to A$
    \end{itemize}
    such that $g_n=  g'_n\circ h_n$ for all $n\in\N$.
\end{lem}

Note that the $g'_n$ are thrifty (as they are injective), and that
the $h_n$ are in $\C_I$ (as $A$ has width 1).

\begin{proof}
    For all $n\in\mathbb{N}$ and every $c\in\ran(g_n)$, pick a tuple $u^{c,n}\in X^M$ in such a way
    that for distinct $(d,i), (e,j)\in Y\mult \mathbb{N}$, the components of the corresponding tuples $u^{d,i}, u^{e,j}$
    have no $y$-coordinates in common; this is possible as
    $g_n\inv[c]$ is unbounded for all $n\in\mathbb{N}$ and all $c\in \ran(g_n)$. Set $A_n$ to consist of all tuples chosen for $g_n$, and let
    $A=\bigcup_n A_n$. Let $g'_n$ be the restriction of
    $g_n$ to $A_n$. Clearly, $A$ has width $1$ and the $g'_n$ are injective
    on their respective domains $A_n$. The existence of the $h_n$ is then trivial.
\end{proof}

The core of Part~1 of our proof is the following lemma.

\begin{lem}\label{lem.strong}
    Let $g:X^M\to Y$ be a partial function and fix a set $S\subseteq M$. Then
    we can write $g$ as $g=g'\circ h$, where $g'\subseteq g$, $g'\cc$ is
    thrifty for each $c\in X^S$, and $h: X^M\To X^M$ is in $\C_I$.
\end{lem}

Before we prove the lemma, we observe that it implies all we need in
this section.

\begin{cor}\label{cor.decomposition}
    Let $g:X^M\to Y$ be a partial function. Then we can write $g$ as
    $g=g'\circ h$, where $g'\subseteq g$, $g'$ is hereditarily  thrifty,
    and $h$ is in $\C_I$.
\end{cor}

\begin{proof}[Proof of the corollary]
    We enumerate all subsets $S$ of $M$ as
    $S_1, \ldots, S_{k}$.  Applying Lemma \ref{lem.strong} repeatedly, we inductively define a sequence
    $g=g^0\supseteq g^1\supseteq \cdots \supseteq g^k$ such that:
    \begin{itemize}
        \item $g^i\cc$ is thrifty for each $c\in X^{S_i}$
        \item $g^{i-1}= g^i\circ h^i$ for some function $h^i: X^M\To X^M$ in $\C_I$.
    \end{itemize}
    We have $g = g^0 = g^1 \circ h^1 = g^2\circ h^2\circ h^1 =
    \cdots =g^k\circ (h^k\circ\cdots \circ h^1)$. As $g':=g^k\subseteq g^i$ for all $i$, we see that $g'$ is hereditarily
    thrifty. Clearly, $h:=h^k\circ\cdots \circ
    h^1$ is in $\C_I$.
\end{proof}

\begin{defn}
    Let $S\subseteq M$, $T:=M\setminus S$, and $c\in X^S$.
    \begin{itemize}
    \item    For any set $A\subseteq X^T$  we write $c*A$ for the set $\{c\cup z:
    z\in A\}\subseteq X^M$.
    \item
    For any partial function $g:X^T\to Y$ we write $c*g$ for the
    function with domain $c*\dom(g)$ defined by $(c*g)(c\cup z) = g(z)$.
    \item
    If $Y=X^T$, then we write $c\#g$ for the partial function from $X^{M}
    $ to $X^{M}$ mapping each $c\cup z\in c*\dom(g)$ to $c\cup g(z)$.
    \end{itemize}
\end{defn}

The following lemma shows how to calculate with the operators just
defined; we leave its straightforward verification to the reader.

\begin{lem}\label{lem.grauslich} Let $M, S, T, $ and $c$ as in the preceding definition.
\begin{enumerate}
\item
   For all partial functions $f: X^T \To Y$ and all partial $g: X^T\To X^T$,  $c*(f\circ g )   = (c*f) \circ (c\# g)$.
\item
   For every partial function $g:X^T\to Y$, $(c*g)\cc = g$.
\item
   For every partial function $g:X^M\to Y$,
    $ g = \bigcup_{c\in X^S}  c*(g\cc)$.
    \\
    Moreover: Whenever $(g_c:c\in X^S)$ is a family of partial
    functions from $X^T$ to $Y$ with $g = \bigcup_{c\in X^S}  c*g_c$, then
    we must have $g_c = g\cc$  for all $c\in X^S$.
\end{enumerate}
\end{lem}

We are now ready to finish Part~1 and provide the proof of its main
lemma.

\begin{proof}[Proof of Lemma \ref{lem.strong}]
    Using Lemma \ref{lem.grauslich}, we split $g$ into its ``components'' $g\cc: X^T\To Y$, and use Lemma
    \ref{countable} to deal with those functions; that is, we write $ g = \ccup c*g\cc$.
    By Lemma~\ref{lem.splitting}, each function $g\cc$ can be written as $t_c\cup w_c$, where $t_c$ is
    thrifty, $w_c$  is wasteful, and $t_c$ and $w_c$ have disjoint domains.

    By Lemma~\ref{countable}, we can find a set $A$ of width 1,
    injective functions $w'_c\subseteq w_c$ whose domains are disjoint
    subsets of $A$, and partial functions $h_c$ with $\ran(h_c)
    \subseteq A$, such that $w_c = w'_c \circ h_c$, for all $c\in X^S$. We write $t_c$ as the composition $t_c = t_c \circ i_c$,
    where $i_c$ is the identity function on $\dom(t_c)\subseteq X^T$.

    Now the domains of $i_c$ and $h_c$, as well as those of $t_c$ and $w_c'$, are disjoint, so Fact~\ref{fact.union} implies that for all $c\in X^S$,
    $$
        g\cc = t_c\cup w_c = (t_c\circ i_c) \cup (w'_c\circ h_c)
        \subseteq (t_c\cup w'_c) \circ (i_c\cup h_c).
    $$

    Therefore, by Lemma~\ref{lem.grauslich},
    $$
        c*g\cc  \subseteq c*[(t_c\cup w'_c) \circ (i_c\cup h_c)]  =
        (c*(t_c\cup w'_c)) \circ( c\# (i_c\cup h_c)  ).
    $$
    This, together with Fact~\ref{fact.union}, allows us to calculate:
    $$g = \ccup c*g\cc \subseteq \ccup
    (c*(t_c\cup w'_c)) \circ( c\# (i_c\cup h_c)  ) \subseteq
    $$
    $$\subseteq
                \bigg(\ccup
    (c*(t_c\cup w'_c))
                \bigg)
            \circ
                \bigg(\ccup
    c\# (i_c\cup h_c)  )
            \bigg)   =: g' \circ h'.
    $$
    Now for all $c\in X^S$, $g'\cc = t_c\cup w'_c$ is thrifty, since $t_c$ is thrifty and since $w'_c$ is injective (and hence thrifty), and by Lemma~\ref{lem.finiteUnions}.
    Since by Fact~\ref{fact.sub}
    there is $h\subseteq h'$ such that $g=g'\circ h$, it remains to see that $h'$
    is in $\C_I$. A quick check of the
    definitions shows that $h'$ is the union of $\ccup c\# i_c$ with $\ccup c\#
    h_c$; by Lemma \ref{lem.finiteUnions}, it suffices to check that each of these unions is in $\C_I$. All components of the first union are projections, thus
    certainly in $\C_I$. For the second union, those components of with index in $S$ are
    just projections as well; the components with index
    in $T$, on the other hand,
    have range in a projection of $A$, and hence are elements of $\C_I$ as well.
\end{proof}

\bigskip \subsection*{Part 2: Generating  hereditarily thrifty
functions}

In this section, we fix an arbitrary $f\in\C_J\setminus\C_I$ and
prove the following:

\begin{lem}\label{lem.hereditarilyThriftyAreGenerated}
    Let $q: X^M\To X$ be a hereditarily thrifty partial function in $\C_J$.
    Then $q$ can be written as a term of $f$ and partial functions in $\C_I$.
\end{lem}

Thus, given an arbitrary $g\in\C_J$, we can use
Corollary~\ref{cor.decomposition} to decompose it as $g=q\circ h$,
and by Lemma~\ref{lem.hereditarilyThriftyAreGenerated}, we can write
$q$ as a term of $f$ and partial $\C_I$--functions. By
Lemma~\ref{lem.extension}, we can extend each of the partial
$\C_I$--functions in this term to total functions in $\C_I$, which
yields a representation of $g$ as a term over $\{f\}\cup\C_I$,
finishing our proof.

In order to prove the lemma, we first show that we can assume $f$ to
be unary.

\begin{lem}
    $\{f\}\cup\C_I$ generates a unary operation in
    $\C_J\setminus\C_I$.
\end{lem}
\begin{proof}
    For a set of unary operations $\M\subseteq \Oo$, set
    $$
        \Pol(\M):=\{g\in\O:
        g(g_1,\ldots,g_n)\in \M \text{ for all } g_1,\ldots,g_n\in\M\}.
    $$
    It is a fact (see \cite{BGHP}) and easy to see that for any
    ideal $K$, $\C_K=\Pol(\C_K\cap\Oo)$. Hence, since $f\nin
    \C_I$, there exist unary $g_1,\ldots,g_n\in\C_I$ such that
    $f(g_1,\ldots,g_n)\nin\C_I$. As $f,g_1,\ldots,g_n$ are elements of
    $\C_J$, so is $f(g_1,\ldots,g_n)$. Hence, the latter operation
    witnesses our assertion.
\end{proof}

Referring to the preceding lemma, we assume without loss of
generality that $f$ is unary itself. Hence, there is a set $A\ss X$
of width 1 that $f$ maps onto a set in $J\setminus I$. Since every
function which maps every line $\linen$ into itself is an element of
$\C_I$, we can simplify notation by assuming that all elements of
$A$ have $x$-coordinate $0$.

Fix an injective map $(n,k)\mapsto n\oplus k$ from $\{(n,k): k< n\in
\omega\}$ into a co-infinite subset of $\omega$ (for example
$n\oplus k = n^2+k$). By further permuting of lines and rows we may
without loss of generality assume that

\begin{flushright}
    $f((0|n\oplus k))= (k|n)$\qquad\qquad\qquad\qquad\qquad$(^*)$
\end{flushright}
for all $k< n<\omega$.

For any finite index set $M=\{1,\ldots,m\}$, define
$$P^*(M) = \{ (S,j): S\subseteq M, j\in M\setminus S\}.$$
We will show that if $q: X^M\To X$ is hereditarily thrifty in and in
$\C_J$, then there are partial $\C_I$--functions $Q:X^M\times
X^{P^*(M)} \to X$ and $h^{S,j}:X^{M}\to X$, $(S,j)\in P^*(M)$, such
that
$$
    q(u) =
    Q\bigg(u,\bigl(f (h^{S,j}(u)): (S,j)\in P^*(M)\bigr)\bigg).
$$
Clearly, this is sufficient for the proof of Lemma
\ref{lem.hereditarilyThriftyAreGenerated}. We start by defining the
$h^{S,j}$; to do this, we need the following lemma.

\begin{lem}\label{lem.thriftyImpliesBounded}
    Let $t: X^M\To X$ be a partial function which is thrifty and in $\C_J$. Then $t\inv[\linen]$ is bounded for
    all $n\in\omega$.
\end{lem}
\begin{proof}
    Suppose this is not the case, and write
    $t\inv[\linen]=\{c^0,c^1,\ldots\}$. Since this set is unbounded, we
    can find an infinite $U\subseteq \mathbb{N}$ such that the sequences
    $((c^i_j)^y : i\in U)$ are injective for all $j\in M$ (where $c^i_j$ is the $j$-th component of the $M$-tuple $c^i$).
    Now given some value $d\in\linen$, only finitely many of the
    tuples
    $\{c^i:i\in U\}$ can be mapped to $d$, as $t$ is thrifty.
    Thus, the set $\{c^i:i\in U\}\in I^M$ is mapped by $t$ to an infinite
    subset of $\linen$ and hence to a set
    outside $J$, a contradiction.
\end{proof}

Lemma~\ref{lem.thriftyImpliesBounded} allows us to make the
following definition.

\begin{defn}
    For any
    thrifty partial $\C_J$--function $t:X^M\to X$ define a function
    $K_t:\omega \to\omega $  by
    $$
        K_t(n) =
        \min \{k:  t^{-1}(\omega \times \{n\}) \subseteq B_k^M\}.
    $$
    %
\end{defn}

\begin{defn}
    Let $q:X^M\to X$ be a hereditarily thrifty partial $\C_J$--function, and $(S,j)\in
    P^*(M)$. Write $T:=M\setminus S$; then
    each tuple $u\in X^{M}$ can be written as $u=c\cup z$, $c\in
    X^S$, $z\in X^T$.  Write $z_j$  for the $j$-th component of
    any such $z$.

    We define $h^{S,j}:X^{M}\to X$ as follows.   For any $c\in
    X^S$, $z\in X^T$ we let
    $$
            h^{S,j}(c\cup z) =
                \begin{cases}
                    \bigl(0| K_{q\cc} ( q(u)^y ) \oplus z_j^y\bigr) & \mbox{if $z_j^y < K_{q\cc}(q(u)^y)$}\\
                    \mbox{undefined} & \mbox{otherwise.}
                \end{cases}
    $$
\end{defn}

Observe that since the range of $h^{S,j}$ has width 1, the function
so defined is an element of $\C_I$. We turn to the definition of
$Q$.

\begin{defn}
    For any $q:X^M\to Y$, we define a partial function $Q:X^M\times
    X^{P^*(M)} \to Y$: For any $u = (u_i:i\in M)$ and any $v = (v_{S,j}:
    (S,j)\in P^*(M))$ we let
    $$Q(u,v) =
    \begin{cases}
        q(u),            &\mbox{if $v_{S,j}= f(h^{S,j}(u))$ for all $(S,j)\in P^*(M)$} \\
        \mbox{undefined}              &\mbox{otherwise.} \\
    \end{cases}
    $$
\end{defn}

Clearly, $q(u) = Q\bigg(u,\bigl(f(h^{S,j}(u)): (S,j)\in
P^*(M)\bigr)\bigg)$. Therefore, to obtain a proof of Lemma
\ref{lem.hereditarilyThriftyAreGenerated}, it suffices to show that
if $q$ in $\C_J$, then $Q\in \C_I$. This will be the direct
consequence of the following final lemma.

\begin{lem}\label{lem:mainLemma}
    Assume that $(A_i: i\in M)$
    and $(A_{S,j}: (S,j)\in P^*(M) )$ are sets of width 1.
    Let $A:= X^M\times X^{P^*(M)}$
    be the product of those sets. Then  $Q[A]$ is a set of width
    $\le m!$.
\end{lem}

Observe that if $E\in I$, then it has finite width, so it can be
written as a finite union $E=\bigcup_{i\in N}E_i$ of sets of width
1. Set $R:=M\times P^*(M)$. Now $E^R=\bigcup_{r\in N^R} \prod_{j\in
R} E_{r(j)}$. Since every of the products in the big union is mapped
by $Q$ to a set of width $\leq m!$, the whole union is mapped to a
set of finite width. Hence, $Q\in\C_I$.

\begin{proof}[Proof of Lemma \ref{lem:mainLemma}]

    We show:
    \begin{quote}
        For all $n\in\omega$, and all permutations $\pi:M\to M$ there is at most one tuple $(u,v)\in  A
        $
        with  $Q(u,v) \in \omega\times \{n\}$
        such that
        $u_{\pi(1)}^y\le \cdots\le u_{\pi(m)}^y$.
    \end{quote}

    This clearly implies the assertion. For notational simplicity (but without loss of generality, since the definition of $Q$ did not refer to any
    order of the indices) we will
    prove this only for the special case where $\pi$ is the
    identity; that is, we show that there is at most one tuple
    $(u,v)\in A$ with $u_1^y\leq \cdots\leq u_m^y$ and $Q(u,v)\in
    \omega\times\{n\}$.

    By assumption, all factors $B\subseteq X$ of the product set $A$ have width
    $1$, meaning that their intersection with every line contains at
    most one element; by replacing these factors by supersets (still of width 1), we may assume that
    these intersections are never empty. For every factor $B$, we
    write $B\langle n\rangle$ for the unique $k$ with $(k|n)\in B$.

    We now define inductively tuples $c^0,\ldots,c^m$
    such that
    $c^j\in X^{\{1,\ldots, j\}}$ for all $j\leq m$ and such that for all $1\leq
    j\leq m$, $c^j$ is the extension of $c^{j-1}$ by one coordinate.

    Let $c^0\in X^\emptyset$ be the empty  tuple. To continue, set
    $$
        k_1:= K_q(n),\quad b_1 := A_{\emptyset, 1}\langle k_1\rangle,\quad a_1:=A_1\langle b_1\rangle,
    $$
    and let $c^1$ be the tuple mapping $1$ to $(a_1|b_1)$. For $j\ge 1$, having already defined the tuple $c^{j-1}\in
    X^{\{1,\ldots,
    j-1\}}$, we set
    $$
        k_{j}:= K_{q_{\cup c^{j-1}}}(n), \quad
        b_{j} := A_{\{1,\ldots, j-1\},j} \langle k_{j}\rangle,\quad
        a_j:=A_j\langle b_j\rangle,
    $$
    and let $c^j$ be the tuple extending $c^{j-1}$ which maps  $j$ to
    $(a_j|b_j)$. After $m$ steps, we arrive at $c^m = ( \,(a_1|b_1), \ldots, (a_m|b_m)\,)$.

    We claim that for all $(u,v)\in A$ with $u_1^y\le \cdots
    \le u_m^y$, if $Q(u,v)\in \linen $ then $u_j = (a_j|b_j)$ for all $1\leq j\leq
    m$. This will finish the proof since if these hypotheses
    uniquely determine $u$, then they also determine $v$, as is
    obvious from the definition of $Q$.

    To see the truth of our final claim, take any tuple $(u,v)$ satisfying the hypotheses; so
    $Q(u,v) = q(u) \in \linen $.

    We start by showing $u_1=(a_1|b_1)$. Since
    $q\inv[\linen]\subseteq B_{k_1}^M$, there is some $i\in M$ such that $u_i^y<k_1$; by our assumption $u_1^y\leq\cdots\leq u_m^y$,
    this implies $u_1^y < k_1$. Therefore, the definition of $h^{\emptyset,1}$ yields
    $h^{\emptyset, 1}(u) = (0|K_{q_{\cup \emptyset}}(q(u)^y) \oplus u_1^y)=(0|K_q(n) \oplus u_1^y)=(0|{k_1} \oplus u_1^y)$. Consequently,
    by our assumption $(^*)$ on $f$,
    $f(h^{\emptyset, 1}(u)) = (u_1^y|k_1)$.   Since $Q(u,v)$
    is defined, we must have $v_{\emptyset,1} = f(h^{\emptyset, 1}(u)) = (u_1^y|k_1)$. But $v_{\emptyset,1}\in A_{\emptyset, 1}$, hence
    $u_1^y =
    A_{\emptyset,1}\langle k_1\rangle$. As $u_1= (u_1^x|u_1^y)\in A_1$ we must also
    have $u_1^x = A_1\langle u_1^y\rangle$.  Hence, $u_1 = (a_1|b_1)$.

    We now show that $u_2=(a_2|b_2)$. To this end, we consider the function $q\cco$, a
    function from $X^{\{2,\ldots, m\}}$ into $X$.   We have $q\cco(u_2,
    \ldots, u_n) \in \linen$. Therefore, $q\cco\inv[\linen]\subseteq B_{k_2}^{\{2,\ldots,m\}}$, so we must have $u_2^y < k_2$.
    As above we conclude (in this order):
    \begin{itemize}
        \item $h^{\{1\}, 2}(u) = (0|{k_2} \oplus u_2^y)$
        \item $f(h^{\{1\}, 2}(u)) = (u_2^y|k_2)$
        \item $v_{\{1\},2} = f(h^{\{1\}, 2}(u)) = (u_2^y|k_2)\in A_{\{1\},2}$
        \item  $u_2^y = A_{\{1\},2}\langle k_2\rangle$
        \item $u_2^x = A_2 \langle u_2^y\rangle$
        \item So $u_2 = (u_2^x|u_2^y)=(A_2\langle u_2^y\rangle |A_{\{1\},2}\langle k_2\rangle)=(a_2|b_2)$.
    \end{itemize}

    Continuing inductively in this fashion, one sees that indeed, $u_j=(a_j|b_j)$ for all $1\leq j\leq m$.
\end{proof}


\providecommand{\bysame}{\leavevmode\hbox
to3em{\hrulefill}\thinspace}
\providecommand{\MR}{\relax\ifhmode\unskip\space\fi MR }
\providecommand{\MRhref}[2]{%
  \href{http://www.ams.org/mathscinet-getitem?mr=#1}{#2}
} \providecommand{\href}[2]{#2}

\end{document}